\documentclass[12pt]{article}
\usepackage[centertags]{amsmath}
\usepackage{amsfonts}
\usepackage{amssymb}
\usepackage{latexsym}
\usepackage{amsthm}
\usepackage{newlfont}
\usepackage{graphicx}
\usepackage{listings}
\usepackage{booktabs}
\usepackage{abstract}
\RequirePackage{srcltx}

\date{}

\newlength{\defbaselineskip}
\newcommand{\setlinespacing}[1]%
{\setlength{\baselineskip}{#1 \defbaselineskip}}

\newcommand{\bR}{{\mathbb{R}}}

\newcommand{\N}{{\mathbb{N}}}

\newcommand{\actaqed}{\hfill $\actabox$}
{\medskip\noindent \textit{Proof of #1. }}%
{\actaqed \medskip}

\usepackage{enumerate}
\def \Tr{\mathcal T}

\def \cM{\mathcal M}

\def\Z{\mathbb Z}

\def \T{\mathbb T}

\def \<{\langle}
\def\>{\rangle}

\def \ep{\epsilon}
\def \e{\varepsilon}

\def\al{\alpha}

\def \sp{\operatorname{span}}

\def\ba{\mathbf a}

\def\bx{\mathbf x}

\def\bz{\mathbf z}
\def\bk{\mathbf k}

\def\bs{\mathbf s}

\theoremstyle{plain}

\newtheorem{Theorem}{Theorem}[section]
\newtheorem{thm}{Theorem}[section]
\newtheorem*{theorem*}{Theorem}

\newtheorem{lem}[thm]{Lemma}

\newtheorem{Proposition}{Proposition}[section]
\newtheorem{Remark}{Remark}[section]

\numberwithin{equation}{section}
\theoremstyle{definition}
\newtheorem*{condE}{Condition E}%[section]
\newtheorem*{Example}{Example}%[section]

\newcommand{\be}{\begin{equation}}
\newcommand{\ee}{\end{equation}}

\def\va{\varepsilon}

\def\sa{{\sigma}}

\def\ta{\theta}
\def\al{{\alpha}}

\def\RR{\mathbb{R}}
\def\FF{\mathcal{F}}

\def\Br{\Bigr}
\def\Bl{\Bigl}
\def\f{\frac}
\def\bl{\bigl}
\def\br{\bigr}

\newcommand{\wt}{\widetilde}

  \def\Ga{\Gamma}

  \def\vi{\varphi}

\def\SS{\mathbb{S}}

\def\sub{\substack}

\def\ld{\lambda}

\def\og{\omega}
\def\Ld{\Lambda}

\def\Og{\Omega}
\def\cA{\mathcal{A}}
\def\cH{\mathcal{H}}

\def\NN{\mathbb{N}}

\DeclareMathOperator{\Vol}{Vol}

\begin{document}
%\usepackage{amsmath,amssymb}
%% my poor man's solution to arc notation
%\newcommand{\tarc}{\mbox{\large$\frown$}}
%\newcommand{\arc}[1]{\stackrel{\tarc}{#1}}
%%%
%\pagestyle{empty}

\title{Entropy numbers %in the $L_\infty$ norm
 and  Marcinkiewicz-type discretization theorem
	\footnote{
		The first named author's research was partially supported by NSERC of Canada Discovery Grant RGPIN 04702-15.
		The second named author's research was partially supported by NSERC of Canada Discovery Grant RGPIN 04863-15.
		The third named author's research was supported by the Russian Federation Government Grant No. 14.W03.31.0031.
		The fifth  named
		author's research was partially supported by
		MTM 2017-87409-P,  2017 SGR 358, and
		the CERCA Programme of the Generalitat de Catalunya.}
}

\author{ F. Dai, \, A. Prymak,\, A. Shadrin, \\ V. Temlyakov, \, and  \, S. Tikhonov  }

\newcommand{\Addresses}{{% additional braces for segregating \footnotesize
		\bigskip
		\footnotesize
		
		F.~Dai, \textsc{ Department of Mathematical and Statistical Sciences\\
			University of Alberta\\ Edmonton, Alberta T6G 2G1, Canada\\
			E-mail:} \texttt{fdai@ualberta.ca }

		\medskip
		A.~Prymak, \textsc{ Department of Mathematics\\
			University of Manitoba\\ Winnipeg, MB, R3T 2N2, Canada
			\\
			E-mail:} \texttt{Andriy.Prymak@umanitoba.ca }
		
		\medskip
		A.~Shadrin, \textsc{Department of Mathematics and Theoretical Physics\\
			University of Cambridge\\Wilberforce Road, Cambridge CB3 0WA, UK
			\\
			E-mail:} \texttt{a.shadrin@damtp.cam.ac.uk}

		\medskip
		V.N. Temlyakov, \textsc{University of South Carolina,\\ Steklov Institute of Mathematics,\\ and Lomonosov Moscow State University
			\\
			E-mail:} \texttt{temlyak@math.sc.edu}
		
		\medskip
		
		S.~Tikhonov, \textsc{Centre de Recerca Matem\`{a}tica\\
			Campus de Bellaterra, Edifici C
			08193 Bellaterra (Barcelona), Spain;\\
			ICREA, Pg. Llu\'{i}s Companys 23, 08010 Barcelona, Spain,\\
			and Universitat Aut\`{o}noma de Barcelona\\
			E-mail:} \texttt{stikhonov@crm.cat}
		
}}

\maketitle

\begin{abstract}
	{This paper studies the behavior of the entropy numbers of classes of functions with bounded integral norms from a given  finite dimensional linear subspace.  Upper bounds   of these  entropy numbers  in the uniform norm are obtained and  applied    to establish a Marcinkiewicz type  discretization theorem   for    integral norms  of functions from
		a given finite dimensional subspace. 		
		}
\end{abstract}

\section{Introduction}
\label{I}

%This paper is a follow up to the recent papers \cite{VT158}, \cite{VT159}, \cite{DPSTT} and  \cite{DPTT}.
We start  with   some  necessary notations and    definitions.
  Let $X$ be a Banach space and let  $B_X(g,r)$ denote  the   closed ball $\{f\in X:\|f-g\|\le r\}$ with  center $g\in X$ and radius $r>0$.  For a compact set $A$ in $X$  and a positive number $\e$,   the covering number $N_\e(A, X)$ is defined  as
  \[ N_\va (A,X):=\min\Bl\{ n\in\NN:\  \ \exists\; g^1,\ldots, g^n\in X, \   A\subset \bigcup_{j=1}^n B_X(g^j, \va)\Br\}.\]
%    We denote by $\cN_\e(A,X)$
%the corresponding minimal $\e$-net of the set $A$ in $X$. Namely, $\cN_\va(A,X)$ is a finite subset of $A$ such that $A\subset \bigcup_{y\in \cN_\va(A,X)}B_X(y,\va)$ and  $N_\e(A,X)= |\cN_\e(A,X)|$.
The $\va$-entropy $\mathcal{H}_\va (A; X)$ of the compact set $A$  in $X$  is  defined as
$\log_2 N_\e(A,X)$, and    the entropy numbers $\e_k(A,X)$ of the set $A$ in $X$ are defined  as
\begin{align*}
\e_k(A,X)  :&=\inf \{\e>0: \cH_\e (A; X)\leq k\},\   \  k=1,2,\ldots.
\end{align*}
Note that in our definition here  we do not  require $y^j\in A$, whereas in the  definitions of $N_\e(A, X)$ and $\e_k(A,X)$ in \cite{DPSTT},  this requirement  is  imposed.
However, it is well known (see \cite[p.208]{Tbook}) that these characteristics may differ at most by a factor $2$.

Next, let $\Omega$ be a nonempty set  equipped  with a  probability measure $\mu$.  For  $1\le p< \infty$, let  $L_p(\Omega)$ denote  the real Lebesgue  space $L_p$
 defined with respect to the measure $\mu$ on $\Omega$, and  $\|\cdot\|_p$ the  norm of $L_p(\Og)$.
%In the case of $p=\infty$,  we identify $L_\infty(\Og)$ with the space of all continuous functions on $\Og$.
 Let $X_N$ be an $N$-dimensional linear  subspace of $L_\infty(\Og)$ and  set
$$
X^p_N := \{f\in X_N:\, \|f\|_p \le 1\},\  \ 1\leq p<\infty.
$$
Here and throughout the paper,  the index $N$  always  stands for the dimension of $X_N$, and we assume that each function $f\in X_N$ is defined everywhere on $\Og$.

By discretization of the $L_p$ norm we understand a replacement of the measure $\mu$ by
a discrete measure $\mu_m$ with support on a set $\xi =\{\xi^\nu\}_{\nu=1}^m \subset \Omega$. This means that integration with respect to the measure $\mu$ is  replaced  by evaluation of  an appropriate weighted sum of values  of a function $f$ at a
finite set  of points. This is why we call this way of discretization {\it sampling discretization}.
Discretization is
a very important step in making a continuous problem computationally feasible.
An important   example of a classical discretization problem is the problem of metric entropy (covering numbers, entropy numbers); see
 %. The reader can find fundamental general results on metric entropy in
  \cite[Ch.15]{LGM},  \cite[Ch.3]{Tbook},  \cite[Ch.7]{VTbookMA},  \cite{Carl},
\cite{Schu} and the recent papers  \cite{VT138}, \cite{HPV}.
 Another prominent example of discretization is the problem of numerical integration, which %.  Numerical integration in the mixed smoothness classes
 requires many fundamental results for constructing optimal (in the sense of order) cubature formulas (see, e.g., \cite[Ch.8]{DTU}).

  There are  different ways to discretize: use coefficients from an expansion with respect to a basis, or more generally, use
 linear functionals. We discuss here the way which uses function values at a fixed finite set of points. Our main interest is the problem of   discretization of the $L_p$ norms of functions from a given finite dimensional subspace. This problem arises  in a very natural way in many applications.
  Indeed, a typical approach to solving a continuous problem numerically -- the Galerkin method --
suggests to look for an approximate solution from a given finite dimensional subspace, while a standard way to measure an error of approximation is an appropriate discretization of an $L_p$ norm, $1\le p\le\infty$.  The first results in this direction were obtained by Marcinkiewicz and
by Marcinkiewicz-Zygmund (see \cite{Z}) for discretization of the $L_p$ norms of the univariate trigonometric polynomials in 1930s. This is why  discretization results of this kind are called the Marcinkiewicz-type theorems.
We now proceed to the detailed presentation.
\\
 {  \bf{Marcinkiewicz problem}.}  We say that a linear subspace $X_N$  of $L_p(\Omega)$, $1\le p < \infty$, admits the Marcinkiewicz-type discretization theorem with parameters $m\in \N$ and $p$ if there exist a set $\{\xi^\nu \in \Omega: \nu=1,\dots,m\}$ and two positive constants $C_j(d,p)$, $j=1,2$, such that for any $f\in X_N$ we have
\be\label{I.1}
C_1(d,p)\|f\|_p^p \le \frac{1}{m} \sum_{\nu=1}^m |f(\xi^\nu)|^p \le C_2(d,p)\|f\|_p^p.
\ee
In the case $p=\infty$ we ask for
\be\label{I.2}
C_1(d)\|f\|_\infty \le \max_{1\le\nu\le m} |f(\xi^\nu)| \le  \|f\|_\infty.
\ee
%We will also use a brief way to express the above properties: the $\cM(m,q)$ theorem holds for  a subspace $X_N$ or $X_N \in \cM(m,q)$.
{\bf Marcinkiewicz problem with weights.}  We say that a linear subspace $X_N$ of the $L_p(\Omega)$, $1\le p < \infty$, admits the weighted Marcinkiewicz-type discretization theorem with parameters $m$ and $p$ if there exist a set of knots $\{\xi^\nu \in \Omega\}$, a set of weights $\{\ld_\nu\}$, $\nu=1,\dots,m$, and two positive constants $C_j(d,p)$, $j=1,2$, such that for any $f\in X_N$ we have
\be\label{1.5}
C_1(d,p)\|f\|_p^p \le  \sum_{\nu=1}^m \ld_\nu |f(\xi^\nu)|^q \le C_2(d,p)\|f\|_p^p.
\ee
%Then we also say that the $\cM^w(m,p)$ theorem holds for  a subspace $X_N$ or $X_N \in \cM^w(m,p)$.
%Obviously, $X_N\in \cM(m,p)$ implies that $X_N\in \cM^w(m,p)$.

The most complete results on sampling discretization are obtained in the case $q=2$.  The problem is basically solved in the case of subspaces of trigonometric polynomials. By $Q$ we denote a finite subset of $\Z^d$, and $|Q|$ stands for the number of elements in $Q$. Let
$$
\Tr(Q):= \left\{f: f=\sum_{\bk\in Q}c_\bk e^{i(\bk,\bx)},\  \  c_{\bk}\in\mathbb{C}\right\}.
$$
In \cite{VT158} it was shown how to derive the following result from the
recent paper by  S.~Nitzan, A.~Olevskii, and A.~Ulanovskii~\cite{NOU}, which in turn is based on the paper of A.~Marcus, D.A.~Spielman, and N.~Srivastava~\cite{MSS}.

\begin{Theorem}\label{NOUth}\textnormal{\cite{VT158}} There are three positive absolute constants $C_1$, $C_2$, and $C_3$ with the following properties: For any $d\in \N$ and any $Q\subset \Z^d$   there exists a set of  $m \le C_1|Q| $ points $\xi^j\in \T^d$, $j=1,\dots,m$ such that for any $f\in \Tr(Q)$
	we have
	$$
	C_2\|f\|_2^2 \le \frac{1}{m}\sum_{j=1}^m |f(\xi^j)|^2 \le C_3\|f\|_2^2.
	$$
\end{Theorem}

Some results are obtained under an extra condition on %the orthonormal basis  $\{u_i(x)\}_{i=1}^N$ of
 $X_N$, which we will  call  Condition E for consistency  with  prior work (see, e.g.,
\cite{DPTT}).
%Some results are obtained under an extra condition on the system fui(x)gNi =1.
% We will call it Condition E to be consistent with the prior work (see, e.g.,
%[4]).

\begin{condE}
%{\bf Condition E.}
 There exists a constant $K_1>0$ such that for all $x\in \Omega$
\be\label{ud5}
w(x):=\sum_{i=1}^N u_i(x)^2 \le K_1 N.
\ee
\end{condE}
The reader can find the following result, which is a slight generalization of the Rudelson's \cite{Rud} celebrated result, in \cite{VT159}.

\begin{Theorem}\label{T5.4} \textnormal{\cite{VT159} } Let $\{u_i\}_{i=1}^N$ be a real  orthonormal system satisfying condition {\text E}.
	Then for every $\ep>0$ there exists a set $\{\xi^j\}_{j=1}^m \subset \Omega$ with
	$$
	m  \le C\frac{K_1}{\ep^2}N\log N
	$$
	such that for any $f=\sum_{i=1}^N c_iu_i$ we have
	$$
	(1-\ep)\|f\|_2^2 \le \frac{1}{m} \sum_{j=1}^m f(\xi^j)^2 \le (1+\ep)\|f\|_2^2.
	$$
\end{Theorem}

 For the general case of $1\leq p<\infty$, it  turns out that  certain estimates of the  entropy numbers $\e_k(X^p_N,L_\infty)$ of the class $X_N^p$ in $L_\infty$-norm  play  a crucial
 role in the proof of   an  Marcinkiewicz-type discretization  theorem for the $L_p$ norm of functions from  the space   $X_N\subset L_\infty(\Og)$. This can be seen from  the  following  conditional result,  proved  recently  in  \cite{VT159} for $p=1$,  and in \cite{DPSTT} for $1<p<\infty$:

\begin{Theorem}\label{T2.1}  \textnormal{\cite{VT159,DPSTT}} Let $1\le p<\infty$. Suppose that $X_N$ is an $N$-dimensional subspace of $L_\infty(\Og)$  satisfying  the condition
	\be\label{I6}
	\e_k(X^p_N,L_\infty) \le  B  (N/k)^{1/p}, \quad 1\leq k\le N
	\ee
	with the   constant $B$ satisfying  that  $B\ge 1$ and $\log_2 (2B) \le C_1(p)N$.
	Then for any $\va\in (0,1)$,  there exist a constant $C(p,\va)$ depending only on $\va$ and $p$  and a set of
	\begin{equation}\label{1-6-b}
	m \le C(p,\va)NB^{p}(\log_2(2N))^2
	\end{equation}
		points $\xi^j\in \Omega$, $j=1,\dots,m$,   such that for any $f\in X_N$
	we have
	\begin{equation}\label{1-6-a}
		(1-\va)\|f\|_p^p \le \frac{1}{m}\sum_{j=1}^m |f(\xi^j)|^p \le (1+\va)\|f\|_p^p.
	\end{equation}
\end{Theorem}

 Theorem \ref{T2.1} was proved in \cite{VT159, DPSTT} for the case $\va=\f12$ only, but the proof there with slight modifications works equally well for $\va\in (0,1)$.
For later applications, we also give the following remark here.
\begin{Remark}\label{rem-1-1}
 It is worthwhile to point out that the estimate  \eqref{I6} for $k=1$ implies the following Nikol'skii type inequality for $X_N$,
\[
\|f\|_\infty \le 4BN^{1/p}\|f\|_p\   \ \text{for any $f\in X_N$,}
\]
while  the estimate  \eqref{I6} for $k= \ N$ implies
$$\e_k(X_N^p, L_\infty)\le 6B2^{-k/N}  \   \  \text{ for $k>N$}.$$
The proofs of these  two facts can be found in \cite{DPSTT}.
\end{Remark}

%\begin{rem}\label{rem-1-1}
%	\begin{enumerate}[\rm (i)]
%		\item  It is worthwhile to point out that the  inequality  \eqref{I6} for $k=1$ implies the following Nikol'skii type inequality for $X_N$,
%		\[
%		\|f\|_\infty \le 4BN^{1/p}\|f\|_p\   \ \text{for any $f\in X_N$,}
%		\]
%		and   the inequality  \eqref{I6} for $1\leq k\leq \ N$ implies
%		$$\e_k(X_N^p, L_\infty)\le 6B2^{-k/N}  \   \  \text{ for $k>N$}.$$
%		\item The condition \eqref{I6} implies the following estimates on the entropy:
%		\be\label{2.5}
%		\cH_\e(X_N^p,L_\infty)  \le N(6B/\e)^p \  \ \text{for any}\  \  \e>0.
%		\ee
%		Note  that \eqref{2.5} also implies \eqref{I6} with the constant $B$ being replaced by $6B$.
%		
%	\end{enumerate}
%\end{rem}

Note  that  bounds for the entropy numbers of function classes are important by themselves and also have important connections to other fundamental problems (see, for instance, \cite[Ch.3]{Tbook}  and \cite[Ch.6]{DTU}). Furthermore, the study
of the entropy numbers is a  highly nontrivial and intrinsically interesting subject.

The aim of this paper is twofold. Firstly, we  conduct a detailed study of the entropy numbers $\e_k(X_N^p, L_\infty)$ of $X_N^p$  in the $L_\infty$ norm for $1\leq p\leq 2$. Secondly, we shall   apply  Theorem \ref{T2.1} and the obtained estimates of  the  entropy numbers   to obtain  a Marcinkiewicz-type discretization theorem for the $L_p$  norms of functions from the space $X_N$.
 In this paper, we will focus mainly  on the case of $1\leq p\leq 2$, where rather complete results can be obtained.  It turns out that there are  significant differences  between the cases $1\leq p\leq 2$ and $p>2$.
 %We will come back to these problems for the case $p>2$ in a follow-up paper.
 The  main results  of this paper will be summarized in the next section.
We present a detailed discussion of these results in Section \ref{sec:8}.

Throughout this paper, the letter $C$ denotes  a general positive
constant depending only on the parameters indicated as arguments or subscripts, and we will use the notation $|A|$ to denote the cardinality of a finite set $A$.

\section{Main results}\label{sec:2}

In this section, we  shall  summarize our  main results on the  entropy numbers $\e_k(X_N^p, L_\infty)$  and the  Marcinkiewicz discretization of $L_p$ norms of functions from $X_N$.  As stated in the introduction, we only deal with  the case $1\leq p\leq 2$  in this paper.

Firstly, we prove  the following estimates of the entropy numbers. % in Section~\ref{sec:3}.
\begin{thm} \label{thm-4-1}	Assume that  $X_N$ is  an $N$-dimensional subspace of $L_\infty(\Og)$ satisfying the following two conditions:
	\begin{enumerate}
		\item [{\bf \textup{(i)} }] There exists a constant $K_1>1$ such that
		\begin{equation}\label{4-1}
		\|f\|_\infty \leq (K_1 N)^{\f12}\|f\|_2,\   \ \forall f\in X_N.
		\end{equation}

		\item [{\bf \textup{(ii)}}] There exists a constant $K_2>1$  such that
		\begin{equation}\label{2-2}
		\|f\|_\infty \leq K_2 \|f\|_{\log N},\   \ \forall f\in X_N.
		\end{equation}
	\end{enumerate}
Then for  each $1\leq p\leq 2$,  there exists a constant $C_p>0$ depending only on $p$ such that
	\begin{equation}\label{4-2-3}
	\e_k (X_N^p, L_\infty) \leq C_p (K_1K_2^2\log N)^{\f 1p} \begin{cases}
	\bigl(\f Nk\bigr)^{\f 1p},&\  \ \text{if $1\leq k\leq N$},\\
	2^{-k/N}, &\   \ \text{if $k>N$}.
	\end{cases}
	\end{equation}
%	where the constant  $C_{p, K}$ depends only on $p$ and $K_2, K_3$.  Moreover, there exists a universal constant $C>0$ such that
%	$$ C_{1,K}=C e^{K_2/2} K_3,\   \   \ C_{2,K}=CK_2 K_3.$$
\end{thm}

%	Then for  $1\leq p\leq 2$, 	
%\begin{equation}\label{4-2-3}
%\cH_{\va} (X_N^p, L^\infty) \leq C_{p,K}  \f { N \log N} {\va^p},\    \   \  \forall \va>0,
%\end{equation}

  The key ingredient in the proof of  Theorem \ref{thm-4-1} is the following new  inequality on $\e$-entropy, which appears to be of independent interest (see Lemma \ref{L1L2ineq}):
for $1\le p< 2<q\leq \infty$ and 	$\ta:=(\f 12-\f 1q)/(\f 1p-\f1q)$, we have
	\begin{equation}\label{0-9a}
	\cH_{\va} (X_N^p; L_{q}) \le \sum_{s=0}^\infty \cH_{ 2^{-3}a^{s-1} \va^{\ta} } (X_N^2; L_{q})+\cH_{\va^\ta} (X_N^2; L_{q}),\   \  \va>0,
	\end{equation}
	where $a=a(\ta) =2^{\f \ta {1-\ta}}$.
	Indeed, using  inequality \eqref{0-9a}, and slightly modifying the proof of Theorem \ref{thm-4-1}, we can deduce  the following more general estimates under  the conditions \eqref{4-1} and \eqref{2-2}:  for $1\leq p\leq 2<q\leq \infty$,
		\begin{equation}\label{8-1a}
		\e_k(X_N^p; L_q)\leq C_{p,q} (K_1 K_2^2\log N) ^{\f 1p-\f1q} \begin{cases} \bl( \f N k \br)^{\f 1p-\f1q}, &  \  \  \ \text{if $1\leq k\leq N$};\\
		2^{-k/N}, &\  \ \text{if $k\ge N$}.	
		\end{cases}
		\end{equation}
The detailed proof of \eqref{8-1a} will be given in the last section, Section \ref{sec:8} (see Theorem~ \ref{thm-8-1}).

Secondly, in Section~\ref{sec:4} we prove the following  Marcinkiewicz discretization theorem for  the $L_p$ norms of functions from $X_N$.

\begin{thm}\label{thm-6-2}
	Let $X_N$ be an $N$-dimensional subspace of $L_\infty(\Og)$ satisfying the condition \eqref{4-1} with $\log K_1\leq \al \log N$ for some constant $\al>1$.
	 If   $1\leq p\leq 2$  then for any $\va\in (0, 1)$,
	there exists   a set of 	$$m\leq C_p(\al,\va) K_1   N \log^3 N$$ points   $\xi^1, \ldots, \xi^m\in\Og$
	such that
	\begin{equation}\label{MZ-1-a}
(1-\va) \|f\|_p^p\leq \f 1m \sum_{j=1}^m |f(\xi^j)|^p \leq (1+\va) \|f\|_p^p,\  \ \forall f\in X_N, \end{equation}
where the constant $C_p(\al,\va)$ depends only on  $p$, $\va$ and $\al$.
\end{thm}

 Our  proof of   Theorem \ref{thm-6-2} relies on
the estimates of the entropy numbers in  Theorem \ref{thm-4-1} and   the conditional theorem, Theorem \ref{T2.1}.  However, special efforts are also required as   condition \eqref{2-2} is assumed in  Theorem \ref{thm-4-1} but not assumed in Theorem \ref{thm-6-2}.

Finally, we prove 	the following unconditional {\it weighted} discretization result for the  $L_p$ norms of functions from $X_N$ in   Section \ref{sec:5}.

\begin{thm}\label{thm-2-1} Given  $1\leq p\leq 2$,      an arbitrary  $N$-dimensional subspace $X_N$  of $L_p(\Og)$ and any $\va\in (0, 1)$,  there exist   $\xi^1,\ldots, \xi^m\in\Og$ and   $\ld_1,\ldots, \ld_m>0$  such that  $m\leq C_p(\va) N\log^3 N$ 	 and
	\begin{align}\label{weightedMZ}
	(1-\va)	\|f\|_{p} \leq   \Bl(\sum_{j=1}^m\ld_j |f(\xi^j)|^p\Br)^{\f1p}\leq (1+\va) \|f\|_{p},\   \  \forall f\in X_N.
	\end{align}	
\end{thm}

 The proof of Theorem \ref{thm-2-1}   is based on  Theorem \ref{thm-6-2} and a change of density argument from functional analysis. While the weights  $\ld_j$ in \eqref{weightedMZ} are in general not equal,
Theorem \ref{thm-2-1} is applicable to every $N$-dimensional subspace $X_N$ of $L_p$  without any  additional assumptions.

We conclude this section  with a few remarks on our results. Firstly, we point out  that   condition     \eqref{4-1} is equivalent to  Condition~{ E} in the introduction. This can be seen from the following well-known result.
% hhh Indeed, the following statement is well known:
\begin{Proposition}\label{2P1} Let $X_N$ be an $N$-dimensional subspace of $L_\infty$.
	Then for any orthonormal basis $\{u_i\}_{i=1}^N$ of $X_N\subset L_2$ we have that  for $x\in \Omega$
\begin{equation}\label{vsp}
	\sup_{f\in X_N; f\neq 0}|f(x)|/\|f\|_2 = w(x) := \left(\sum_{i=1}^N u_i(x)^2\right)^{1/2}.
\end{equation}
\end{Proposition}
Clearly, \eqref{vsp} follows from
	$$
	\sup_{f\in X_N; f\neq 0}|f(x)|/\|f\|_2= \sup_{\sub{(c_1,\dots,c_N)\in\RR^N\\
			\sum_{i=1}^N c_i^2 = 1}}\left|\sum_{i=1}^N c_i u_i(x)\right| = w(x)
	$$
for any	$
	f(x) = \sum_{i=1}^N c_i(f)u_i(x)$,\  $c_i (f)\in\RR,$  \ $1\leq i\leq N.$
%\end{proof}

Secondly, note  that  inequality  \eqref{4-1}  also    implies the following Nikolskii inequalities for $X_N$:
\begin{equation}\label{2-5-a}
\|f\|_q\leq (K_1N)^{\f 1p-\f 1q} \|f\|_p,\    \  \forall f\in X_N,\   \ 1\leq p\leq 2, \  \ p<q\leq \infty.
\end{equation}
Indeed,  if $1\leq p\leq 2$ and $q=\infty$, then  using  \eqref{4-1}, we obtain   that for any $f\in X_N$,
\[ \|f\|_\infty \leq (K_1 N)^{\f 12} \|f\|_2\leq (K_1 N)^{\f12} \|f\|_p^{\f p2}\|f\|_\infty ^{1-\f p2},\]
which implies
\begin{equation}\label{2-9-b}
\|f\|_\infty \leq (K_1N)^{\f 1p} \|f\|_p.
\end{equation}
If $1\leq p\leq 2$, and $q>p$, then using \eqref{2-9-b}, we have that
\[ \|f\|_q \leq \|f\|_\infty ^{1-\f pq} \|f\|_p^{\f pq} \leq (K_1N)^{\f 1p-\f 1q} \|f\|_p.\]

%		Secondly, we note that  the condition \eqref{2-2} can be easily deduced from    the following Nikolskii type inequalities with some parameter $\al\ge 0$:
%		$$ \|f\|_\infty \leq K_2 e^{-\al} N^{\al/q}\|f\|_q,\   \   \  f\in X_N,\   \  2<q< \infty.$$
	 Finally, we point out that  despite the fact that Theorem \ref{thm-6-2}, the Marcin\-kie\-wicz discretization theorem,   holds without  condition \eqref{2-2},  the entropy  number estimates  \eqref{4-2-3} in Theorem \ref{thm-4-1} are  no longer  true if 	
	 \eqref{2-2} is not assumed. This can be seen from  the following example, which was kindly   communicated to us by B. Kashin.\\
	
	\begin{Example}
\    For each $k\in \NN$, let  $r_k: [0,1]\to\{1,-1\}$ denote   the $k$-th Rademacher  function defined by
	 $$r_k(t):=\textup{sign} \bigl(\sin (2^{k+1} \pi t)\bigr),$$
	 where $\textup{sign} (x)=\begin{cases}
	 1, &\  \ \text{if $x\ge 0$},\\
	 -1,&\ \ \text{if $x<0$}.
	 \end{cases}$
	 Then  $\{r_k\}_{k=1}^\infty$  is an orthonormal system with respect to the Lebesgue measure $d\mu(t)=dt$ on $\Og=[0,1]$.
	 Let
	$$X_N:=\sp \{r_j:\  \ 1\leq j\leq N\}.$$  For each $f_{\ba}=\sum_{j=1}^N a_j r_j\in X_N$ with  $\ba:=(a_1,\ldots, a_N)\in\RR^N$, we have
	\begin{equation}\label{2-5}
	\max_{t\in [0,1]} |f_{\ba}(t)|=\|\sum_{j=1}^N a_j r_j \|_\infty =\sum_{j=1}^N |a_j|,
	\end{equation}
	where the last step uses the fact that
	$$\mu\Bl\{t\in [0,1]:   r_k(t)=\textup{sign} (a_k)\  \ \text{for all $k=1,2,\ldots, N$} \Br\} = 2^{-N}>0.$$
	On the one hand, since the orthonormal system $\{r_k\}_{k=1}^\infty$ satisfies Condition E,   the space $X_N$  satisfies   condition \eqref{4-1}.  On the other hand, however,  \eqref{2-5} implies that  the space $X_N$ does not satisfy   condition \eqref{2-2}.  	Let us show  that the estimates  \eqref{4-2-3} do not hold for any $1\leq p<2$. To this end, let $\wt \ell_q^N$ denote   the space $\RR^N$  equipped with the norm
	\[ \|x\|_{q}:= ( \sum_{j=1}^N |x_j|^q)^{\f 1q},\   \ x=(x_1,\ldots, x_N) \in\RR^N.\]
	By monotonicity of the $L_p$ norms, $X_N^2\subset X_N^p$ for each $1\leq p\leq 2$. Thus, using \eqref{2-5}, we have that
	\begin{align*}
	N_\va (X_N^p, L_\infty)\ge 	N_{\va} (X_N^2, L_\infty)=N_{\va}(B_2^N,\wt \ell_1^N),\   \ 1\leq p<2,\  \ \va>0,
	\end{align*}
	where $B_q^N:=\{x\in\RR^N:\  \ \|x\|_q\leq 1\}$.
	By the standard volume comparison argument, we obtain
	\begin{align*}
	N_{2\va}(B_2^N,\wt \ell_1^N)&          \ge   \frac {\Vol (B_2^N)}{2^N\va^N \Vol (B_1^N)}= \frac {\pi^{\f N2} }{ \Ga (\f N2+1)} \frac { N!} {4^N \va^N } = \frac {\pi^{\f {N-1}2}  \Ga (\f {N+1}2)   }{2^N\va^N}.
	\end{align*}
	This together with  Stirling's formula implies   that for    $1\leq p<2$ and $\va>0$,  \begin{align*}
	\cH_{\va} (X_N^p; L_\infty)&=\log_2 N_\va (X_N^p, L_\infty) \ge
	 \log_2 \Bl[\Bl( \f {\sqrt{\pi}} {2\va}\Br)^N \Bl( \f N{2e}\Br)^{\f N2}\Br]-C\\
	 &\ge N \log_2 \f{\sqrt{N}} \va- C N,
	\end{align*}
	where $C>0$ is an  absolute  constant.  Thus, there exists an  absolute constant  $c_0\in (0,1)$  such that for any $1\leq p<2$ and $0<\va <c_0\sqrt{N}$,
		\begin{equation*}
	\cH_{\va} (X_N^p; L_\infty)\ge  N,
	\end{equation*}
	which in turn implies that for $1\leq k< N$,
	\begin{equation}
	\e_k(X_N^p, L_\infty) \ge c_0 \sqrt{N}.
	\end{equation}
This means that  if   $ N^{\al}\leq k_N\leq N$  for some parameter   $1-\f p2<\al<1$, then
\[ \liminf_{N\to\infty} \f {\sqrt{N}} { \Bl(\f {N\log N} {k_N} \Br)^{\f 1p}}=\infty.\]

	\end{Example}

\section{Proof of Theorem \ref{thm-4-1}}\label{sec:3}

This section is devoted to the proof of  the estimates \eqref{4-2-3} of the entropy numbers $\e_k(X_N^p,L_\infty)$ for $1\leq p\leq 2$.  By the definition of the entropy numbers and  Remark \ref{rem-1-1}, it suffices to show that for $1\leq p\leq 2$,
	\begin{equation}\label{4-2-3b}
	\cH_{\va} (X_N^p; L_\infty) \leq C_p K_1K_2^2 \f { N \log N} {\va^p},\    \   \  \forall \va>0,
	\end{equation}
	where the constant $C$ depends only on $p$.
We divide the proof of \eqref{4-2-3b} into two different  cases: $p=2$ and $1\leq p<2$. The estimate \eqref{4-2-3b}  for $p=2$ is essentially known (see, e.g., \cite{Bel}),  but for the sake of completeness, we will summarize its proof  in Section~\ref{subsection:3.1}.  The proof of  \eqref{4-2-3b} for the remaining    case $1\leq p<2$  will be given  in Section  \ref{subsection:3.3}.

	\subsection{ Case 1.  $p=2$.} \label{subsection:3.1}
	
	Let  $\mathbb{S}^{N-1}$ denote the unit sphere of the Euclidean space $\RR^N$ equipped with the surface Lebesgue measure  	 $\sa$  normalized by $\sa(\SS^{N-1})=1$.  Given an $N$-dimensional normed linear space $X=(\RR^N,\|\cdot\|_X)$, let  $B_X:=\{x\in X:\  \|x\|_X\leq 1\}$ and define
	$$ M_X:=\int_{\mathbb{S}^{N-1}} \|x\|_X\, d\sa(x).$$  We also denote by  $X^\ast$  the dual $(\RR^N, \|\cdot\|_{X^\ast})$  of $X=(\RR^N,\|\cdot\|_N)$.
		
		We need  the following lemma, which can be  found in Lemma 2.4 and Propositions 4.1 and 4.2 of \cite{BLM}.
	\begin{lem}\label{lem-1-3} \textnormal{\cite{BLM} }     Let  $X$ denote  the space $\RR^N$ endowed with some  norm $\|\cdot\|_X$.
		 Then the following statements hold:
		\begin{enumerate}[\rm \bf (i)]
			\item  For $0<\va\leq 1$,
			\begin{equation}\label{3-2-b}
			  N\log \f 1\va \leq \mathcal{H}_\va (B_X, X) \leq N\log (1+\f 2\va).
			\end{equation}
			\item There exists a universal constant $C>0$ such that
			\begin{equation}\label{4-6} \mathcal{H}_\va ( B_X, \RR^N) \leq C N \Bl( \f {M_{X^\ast}}{\va}\Br)^2\   \  \text{and}\    \  	 \mathcal{H}_\va (B_2^N, X) \leq  	C N\Bl( \f {M_X}\va \Br)^2,\qquad \end{equation}
			where $B_2^N$ denotes the Euclidean unit ball of $\RR^N$.

		\end{enumerate}
	\end{lem}

Clearly, for the proof of  estimate \eqref{4-2-3b} for $p=2$, it is enough to show the following lemma.

\begin{lem}\label{lem-4-2} If $X_N$ satisfies  condition \eqref{4-1},  then  for $2\leq q<\infty$,
	\begin{equation}\label{4-7}
	\mathcal{H}_\va  (X_N^2, L_q)\leq  CK_1 N q   \va^{-2},\   \ \va>0.
	\end{equation}
	If $X_N$ satisfies both  conditions \eqref{4-1} and \eqref{2-2}, then
	\begin{equation}\label{infinity}
	\mathcal{H}_\va  (X_N^2, L_\infty)\leq C K_1 K_2^2 \va^{-2} N\log N,\   \ \va>0.
	\end{equation}
\end{lem}

\begin{proof}
		By the rotation invariance of the measure $d\sa$ on $\SS^{N-1}$, we have that for any $ x\in\RR^N$ and $1\leq q<\infty$,
		\begin{align}
		\Bl( &\int_{\SS^{N-1}} |x\cdot y|^q\, d\sa(y)\Br)^{\f1q}=\|x\|_2\Bl( \f {2\Ga(\f N2)}{\Ga (\f 12) \Ga (\f {N-1}2)}\int_0^1 x_1^q (1-x_1^2)^{\f {N-3}2}\, dx_1\Br)^{\f 1q}\notag\\
		&=\Bl( \f {\Ga (\f N2) \Ga(\f {q+1}2)} {\Ga(\f12) \Ga(\f {N+q}2) }\Br)^{\f1q}\|x\|_2
		 \sim \f {\sqrt{q} } {\sqrt{N+q}}\|x\|_2\label{3-5-a}
		\end{align}
		with absolute constants of equivalence.
	Using  \eqref{3-5-a} and  \eqref{4-1} (or the equivalent Condition E), we obtain that for  an orthonormal basis $\{\vi_j\}_{j=1}^N$  of $X_N\subset L_2(\Og)$, and  $2\leq q<\infty$,
	\begin{align}
M_{q, X_N} &:= \int_{\SS^{N-1}} \Bl\| \sum_{j=1}^N \xi_j \vi_j\Br\|_{ L_q (d\mu)} \, d\sa_N (\xi)\notag\\
&\leq
\left(\int_{\SS^{N-1}} \Bl\| \sum_{j=1}^N \xi_j \vi_j\Br\|^q_{ L_q (d\mu)} \, d\sa_N (\xi)\right)^{1/q} \le C \sqrt{K_1q}.\label{4-5}	\end{align}
	It then follows by the second inequality in \eqref{4-6} that for $1\leq q<\infty$ and $\va>0$,
	\begin{equation}\label{3-8-0}
	\cH_\va (X_N^2; L_q) \leq C N \Bl( \f {M_{X,q}}\va\Br)^2 \leq C K_1 q N\va^{-2}.
	\end{equation}
This proves  estimate \eqref{4-7}. Finally, by  \eqref{2-2}, we have
\[ \cH_\va (X_N^2; L_\infty) \leq \cH_{\va/K_2} (X_N^2; L_{\log N}),\]
which, using  \eqref{4-7} with $q=\log N$, leads to the estimate  \eqref{infinity}.	
	\end{proof}

  \subsection{Case 2.   $1\leq  p<2$.}\label{subsection:3.3}

  In this subsection, we shall prove  \eqref{4-2-3b}   for  $1\leq p<2$.    Our proof relies on the following %lemma.

  \begin{lem}\label{L1L2ineq}
  	For $1\le p< 2<q\leq \infty$ and 	$\ta:=(\f 12-\f 1q)/(\f 1p-\f1q)$, we have
  	\begin{equation}\label{0-9}
  	\cH_{\va} (X_N^p; L_{q}) \le \sum_{s=0}^\infty \cH_{ 2^{-3}a^{s-1} \va^{\ta} } (X_N^2; L_{q})+\cH_{\va^\ta} (X_N^2; L_{q}),\   \  \va>0,
  	\end{equation}
  	 where $a=a(\ta) =2^{\f \ta {1-\ta}}$.

  \end{lem}

For the moment, we take Lemma \ref{L1L2ineq}  and proceed with the proof of \eqref{4-2-3b}. Using Lemma \ref{L1L2ineq} with $q=\infty$ and $\ta =p/2$, we obtain
	\begin{equation*}
\cH_{\va} (X_N^p; L_\infty) \le \sum_{s=0}^\infty \cH_{ 2^{-3}a^{s-1} \va^{p/2} } (X_N^2; L_{\infty})+\cH_{\va^{p/2}} (X_N^2; L_{\infty}),
\end{equation*}
which in light of  \eqref{infinity} is bounded above by
\begin{align*}
\leq  CK_1 K_2^2 N \log N \va^{-p}  \sum_{s=0}^\infty  a^{-2s} \leq C\va^{-p} K_1 K_2^2 N \log N.
\end{align*}
This proves \eqref{4-2-3b} for $1\leq p<2$.

It remains to prove Lemma  \ref{L1L2ineq}.

  \begin{proof}[Proof of Lemma \ref{L1L2ineq}]
  	  	We use the inequality
  	\[ \cH_{\va} (X_N^p; L_q) \leq \cH_{\va^{1-\ta}} (X_N^p; L_2)+\cH_{\va^\ta}(X_N^2; L_q).\]
  Thus, setting
   $\va_1:=\va^{1-\ta}$,  we reduce to showing that
  	\begin{equation}\label{4-16}
  	\cH_{\va_1} (X_N^p; L_2) \leq \sum_{s=0}^\infty \cH_{ 2^{-3}a^{s-1} \va^{\ta} } (X_N^2; L_{q}).
  	\end{equation}
%Following the idea %    	We get the idea from \cite{BLM},
  	It will be shown that for $s=0,1,\ldots,$
  	\begin{align}\label{3-19-a}
   \cH_{2^{s} \va_1} (X_N^p; L_2) -\cH_{2^{s+1} \va_1} (X_N^p; L_2)\leq \cH_{ 2^{-3}a^{s-1} \va^{\ta} } (X_N^2; L_{q})  ,
  	\end{align}
  	from which \eqref{4-16} will follow by taking  the sum over $s=0,1,\ldots$\\

  	To show \eqref{3-19-a}, for  each nonnegative  integer $s$, let $\FF_s\subset X_N^p$ be a maximal $2^s \va_1$-separated subset of $X_N^p$ in the metric $L_2$; that is, $\|f-g\|_2\ge 2^s\va_1$ for any two distinct functions $f, g\in \FF_s$, and $X_N^p\subset \bigcup_{f\in\FF_s} B_{L_2} (f, 2^s\va_1)$.  Then
  	\begin{equation}\label{4-17}
  	\cH_{2^s \va_1}(X_N^p; L_2) \leq  \log_2 |\FF_s| \leq \cH_{2^{s-1}\va_1}(X_N^p; L_2).
  	\end{equation}
  	Let   $f_s\in \FF_{s+2}$  be such that
  	$$\Bl| B_{L_2}(f_s, 2^{s+2} \va_1)\cap \FF_s\Br|=\max_{f\in\FF_{s+2}} \Bl| B_{L_2}(f, 2^{s+2} \va_1)\cap \FF_s\Br|.$$
  	Since
  	\begin{align*}
  	\FF_s =\bigcup_{f\in \FF_{s+2}}\Bl(B_{L_2}(f, 2^{s+2} \va_1)\cap \FF_s\Br)\subset X_N^p,
  	\end{align*}
  	it follows that
  	\begin{align}\label{4-18}
  	|\FF_s|\leq |\FF_{s+2}| \Bl| B_{L_2}(f_s, 2^{s+2} \va_1)\cap \FF_s\Br|.
  	\end{align}
  	Set
  	$$ \cA_s:=\Bl\{ \f{f-f_s}{2^{s+2} \va_1}:\   \        f\in B_{L_2} (f_s, 2^{s+2} \va_1) \cap \FF_s\Br\}.$$
  	Clearly,  for any $g\in\cA_s$,
  	\begin{equation}\label{3-22-a}
  	\|g\|_2\leq 1,\   \  \|g\|_p\leq  (2^{s+1} \va_1)^{-1}.
  	\end{equation}
  	On the one hand,  using  \eqref{4-17} and \eqref{4-18} implies that
  %	\[ \log_2 |\FF_s| \leq \log_2 |\FF_{s+2}| +\log_2 |\cA_s|,\]
  %	which,  using  \eqref{4-17}, in turn implies that
  	\begin{align}\label{3-23-a}
  	\log_2|\cA_s|&\ge \log_2|\FF_s|-\log_2|\FF_{s+2}|\notag\\
  	& \ge  \cH_{2^{s} \va_1} (X_N^p; L_2) -\cH_{2^{s+1} \va_1} (X_N^p; L_2) .
  	\end{align}
  	On the other hand, since $\f 12 =\f \ta p+\f {1-\ta}q$,  using \eqref{3-22-a} and the fact that $\FF_s$ is $2^s\va_1$-separated in the $L_2$-metric, we have that  for any two distinct  $g',  g\in \cA_s$,
  	\begin{align*}
  	2^{-2} \leq  \|g'-g\|_2\leq \|g'-g\|_p^\ta \|g-g'\|_q^{1-\ta}\leq  \bl (2^{s+1} \va_1\br)^{-\ta} \|g-g'\|_q^{1-\ta},
  	\end{align*}
  	which implies that
  		\[ \|g'-g\|_q \ge 2^{-2}(2^{s-1} \va_1)^{\f \ta {1-\ta}}=2^{-2} a^{s-1} \va^\ta.\]
  			This together with  \eqref{3-22-a} means  that   $\cA_s$ is a $2^{-2} a^{s-1} \va^\ta $-separated subset of $X_N^2$ in the metric $L_{q}$. We obtain
  		\begin{align}\label{3-24-a}
  		 \log_2 |\cA_s| \leq \cH_{ 2^{-3} a^{s-1} \va^\ta  } (X_N^2; L_q).
  		\end{align}
   	Thus, combining  \eqref{3-24-a} with \eqref{3-23-a}, we prove  inequality \eqref{3-19-a}.
  \end{proof}

\section{Proof of Theorem \ref{thm-6-2}}\label{sec:4}

In this section we prove   the Marcinkiewicz discretization theorem for  the $L_p$ norms of functions from $X_N$  with  $1\leq p\leq 2$. More precisely, for a fixed $1\leq p\leq 2$ and each $\va\in (0,1)$,  we shall show that under the condition \eqref{4-1}
with $\log K_1\le \alpha \log N$, there exists  a set of 	 $$m\leq %C(p) K_1K_2^2
C_p(\alpha,\va) K_1 N \log^3 N$$ points   $\xi^1, \ldots, \xi^m\in\Og$
such that
\begin{equation}\label{MZ-1-b}
(1-\va) \|f\|_p^p\leq \f 1m \sum_{j=1}^m |f(\xi^j)|^p \leq (1+\va) \|f\|_p^p,\  \ \forall f\in X_N. \end{equation}
Note that this result cannot be deduced   straightforwardly from  Theorem \ref{thm-4-1} and Theorem \ref{T2.1}  since we do not assume    condition \eqref{2-2} here.

Our proof relies on  two known lemmas. %, the first of which  can be found in \cite[Lemma 2.1]{BLM}.

\begin{lem}\label{lem-2-2}%\cite[Lemma 2.1]{BLM}
{ \textnormal{\cite[Lemma 2.1]{BLM}} }
	Let $\{g_j\}_{j=1}^m$ be independent random variables with mean $0$ on some probability space $(\Og_0, \mu)$, which satisfy$$\max_{1\leq j\leq m} \|g_j\|_{L_1(d\mu)}\leq M_1, \  \  \max_{1\leq j\leq m}  \|g_j\|_{L_\infty(d\mu)} \leq M_\infty$$
	for some constants $M_1$ and $M_\infty$. Then for any $0<\va<1$, the inequality
	\begin{equation*}
	 \Bl|\f 1m \sum_{j=1}^m g_j\Br| \ge \va
	\end{equation*}
	holds with probability $\leq 2 e^{-\f{m\va^2}{ 4M_1M_\infty}}$.
%\footnote{S: $\leq 2 e^{-\f{m\va^2}{ 4M_1M_\infty}}$ instead of $\leq 2 e^{-\f{m\va^2}{ 8M_1M_\infty}}$}
\end{lem}

%The second  lemma can also be found   in \cite{BLM}.

\begin{lem}\label{lem-1-2}%\cite[Lemma 2.5] {BLM}
{ \textnormal{\cite[Lemma 2.5]{BLM}} }
	Let $T: X\to Y$ be a bounded linear map from a normed linear space $(X, \|\cdot\|_X)$ into another normed linear  space $(Y,\|\cdot\|_Y)$.  Let $\va\in (0,1)$ and let $\mathcal{F}$ be an $\va$-net of the unit ball $B_X:=\{x\in X:\  \|x\|_X\leq 1\}$. Assume that there exist constants $C_1, C_2>0$ such that
	$$ C_1 \|x\|_X \leq \|T x\|_Y\leq C_2 \|x\|_X,\   \   \  \forall x\in \mathcal{F}.$$
	Then
	$$ C_1(\va)  \|z\|_X\leq \|Tz\|_Y\leq C_2(\va) \|z\|_X,\   \ \forall z\in X,$$
	where
	\begin{align*}
	C_1(\va):= &C_1(1-\va)-   C_2\va \f {1+\va} {1-\va},\  \
	C_2(\va):=  {C_2 }\f {1+\va}{1-\va}.
	\end{align*}
\end{lem}

To prove  Theorem \ref{thm-6-2}, we start with the following weaker result.

\begin{lem}\label{lem-6-2} Let $1\leq p<\infty$  be a fixed number. Assume that   $X_N$ is  an $N$-dimensional subspace of $L_\infty(\Og)$ satisfying  the following condition  for  some  parameter $\beta >0$ and constant $K\ge 2$:
	\begin{equation}\label{6-4}
	\|f\|_\infty \leq (KN)^{\f \beta p} \|f\|_p,\   \   \forall f\in X_N.
	\end{equation}
	Let $\{\xi_j\}_{j=1}^\infty$ be a sequence of independent random points  selected uniformly from the probability space $(\Og,\mu)$.
	Then   there exists a  positive   constants  $C_\beta$  depending only on $\beta$  such that for any   $0< \va\leq \frac 12$ and    \begin{equation}
	m\ge C_\beta K^\beta \va^{-2}( \log\f 2\va) N^{\beta+1}\log N,\label{4-3-a}
	\end{equation}  the   inequality
	\begin{align}
(1-\va) \|f\|_p^p\leq  \f 1m \sum_{j=1}^m| f(\xi_j)|^p \leq (1+\va)\|f\|_p^p,
	\end{align}holds  with probability
	$ \ge 1-m^{-N/\log K}$.
	
\end{lem}
\begin{proof} The proof is quite standard, and we only sketch the main steps.	
By Lemma \ref{lem-1-3} (i), given $\va\in (0,1)$, there exists  an $\va$-net    $\FF\subset X_N^p$  of $X_N^p$ in the space $L_p$  such that
	$|\FF| \leq \Bl(1+\f 2\va\Br)^N.$
Let   $N_1=KN$ and choose  a universal constant $C_0>1$   so that
\begin{equation}\label{6-5}
|\FF|	\leq   \Bl(1+\f 2\va\Br)^N \leq 2^{-1} \exp\Bl(\f {m\va^2} {16 N_1^\beta}\Br)
\end{equation}
whenever\begin{equation}\label{4-4-a}
	m\ge  C_0 K^\beta  \va^{-2}(  \log \f 2\va)N ^{\beta+1}.
\end{equation}
Next, using  Lemma \ref{lem-2-2}, we  have  that the inequalities
		\begin{equation}\label{7-4-0} \Bl| \f 1 m \sum_{j=1}^m| f(\xi_j) |^p- \int_{\Og} |f|^p\, d\mu\Br| \leq \va\|f\|_p^p,\   \ \forall f\in \FF\end{equation}
		hold  under  condition \eqref{4-4-a}  with probability
	$$ \ge 1-2|\FF|\exp\Bl(- \f {m \va^2} {8N_1^\beta}\Br)\ge 1-\exp\Bl(- \f {m \va^2} {16N_1^\beta}\Br). $$
To complete the proof, we just need to observe that  the function $\f {x}{\log x}$ is increasing on $(e,\infty)$, and hence,  by a straightforward calculation, the condition \eqref{4-3-a}  with a sufficiently large constant $C_\beta$ implies both  \eqref{4-4-a} and
	$$1-\exp\Bl(- \f {m \va^2} {16N_1^\beta}\Br)\ge 1- m^{- N/\log K}.$$
\end{proof}

\begin{proof}[Proof of Theorem \ref{thm-6-2}]
  Given each positive integer $m$ and each  $\bz=(z_1,\ldots, z_{m})\in\Og^{m}$, we define the operator
	$T_{m, \mathbf{z}} : X_N \to \RR^{m}$ by
	$ T_{m, \mathbf{z}} f %=\Bl( T_{m, \mathbf{z}} f(1), \ldots,  T_{m, \mathbf{z}} f(m)\Br)
=(f(z_1), \ldots, f(z_{m})).$ We denote by  $\ell_p^m$   the space $\RR^m$ equipped with the norm
	\[\|x\|_{\ell_p^m} :=\begin{cases}\Bl( \f 1m \sum_{j=1}^m |x_j|^p\Br)^{\f1p},&\  \text{if $1\leq p<\infty$}\\
	\max_{1\leq j\leq m} |x_j|,&\  \text{if $p=\infty$},	
	\end{cases},\   \ x=(x_1,\ldots, x_m)\in\RR^m.\]
	By \eqref{2-5-a} and \eqref{4-1}, we have
	$$\|f\|_\infty \leq (K_1N)^{\f1p} \|f\|_p,\   \ \forall f \in X_N,\   \  1\leq p\leq 2.$$
	Thus, for each fixed $1\leq p\leq 2$, and each $\va\in (0,1)$, by  Lemma \ref{lem-6-2} applied to $\al=1$,   there exists  a vector   $\bz=(z_1,\ldots, z_{m_1})\in\Og^{m_1}$ and a constant $C(\va)>1$  such that    $$C^{-1}(\va) K_1 N^2 \log N \leq m_1\leq C(\va) K_1 N^2\log N,$$  and   the inequalities
	 	\begin{equation}\label{6-7}
	  \Bl|\|f\|_q - \|T_{m_1, \bz} f\|_{\ell_q^{m_1}} \Br|\leq \f \va4\|f\|_q,\   \ \forall f\in X_N^q,\  \ q=2,p
	 \end{equation}
	 hold simultaneously.
	 Consider the $N$-dimensional subspace    $\wt{X}_N :=T_{m_1,\bz} (X_N)$  of $ \ell_p^{m_1}$. Using  \eqref{4-1} and \eqref{6-7}, we have that   for any $f\in X_N$,
\begin{equation}\label{6-7-5}\|T_{m_1,\bz} f \|_{\ell_\infty^{m_1}} \leq \|f\|_{\infty} \leq (K_1N)^{\f12} \|f\|_2 \leq (1+\va) (K_1 N)^{\f12}\|T_{m_1,\bz} f\|_{\ell_2^{m_1}}.
 \end{equation}
On the other hand, since $\log m_1\leq C_\va % K_1 \log N
(\log K_1+\log N)$ and  \[ \|\bx\|_{\ell_{q_2}^{m_1}} \leq m_1 ^{\f 1{q_1} -\f 1{q_2}} \|\bx\|_{\ell_{q_1}^{m_1}},\    \   \forall \bx\in\RR^{m_1},\  \ 0<q_1<q_2\leq \infty,\]
in light of $\log K_1 \le\alpha \log N$,
it follows that
\[ \|\bx\|_{\ell_\infty^{m_1}}\leq e^{c\alpha} \|\bx\|_{\ell_{q_N} ^{m_1}},\    \    q_N=\log N,\  \ \forall \bx\in\RR^{m_1}.\]
This and \eqref{6-7-5} mean that  the $N$-dimensional subspace  $\wt{X}_N$ of $\ell_\infty^{m_1}$ satisfies both the conditions \eqref{4-1} and \eqref{2-2}. It then follows by
 Theorem 	\ref{thm-4-1} that
\begin{equation*}
\e_k (\wt X_N^p, \ell^{m_1}_\infty) \leq
C_p(\va) K_1^{\f 1p} e^{\f {2c\alpha}p}
  (\log N)^{\f 1p} \begin{cases}
\bigl(\f Nk\bigr)^{\f 1p},&\  \ \text{if $1\leq k\leq N$},\\
2^{-k/N}, &\   \ \text{if $k>N$},
\end{cases}
\end{equation*}
where $\wt X_N^p:=\{\bx\in\wt X_N:\  \ \|\bx\|_{\ell_p^{m_1}}\leq 1\}$.
Thus, applying Theorem \ref{T2.1} to the subspace $\wt X_N$ in $\ell_p^{m_1}$,
   we can  find   a  subset  $\Ld \subset\{1,2,\ldots, m_1\}$ with $|\Ld| \leq C_p(\alpha,\va) K_1 N\log^3 N$ such that
	$$ \Bl| \f 1{| \Ld|} \sum_{j\in\Ld} |T_{m_1,\bz} f(j)|^p  -\|T_{m_1, \bz} f\|^p_{\ell_p^{m_1}}\Br|\leq \f \va 4\|T_{m_1, \bz} f\|^p_{\ell_p^{m_1}}. $$
Now using  \eqref{6-7} with $q=p$ and  a sufficiently small parameter $\va\in (0,1)$, we obtain
	$$(1-\va) \|f\|_p^p \leq \f 1 {| \Ld|} \sum_{j\in\Ld} |f(z_j)|^p \leq (1+\va) \|f\|_p^p,\  \ \forall f\in X_N.$$	
	\end{proof}

\section{Proof of Theorem \ref{thm-2-1}}\label{sec:5}
%The main goal in this section is to prove Theorem \ref{thm-2-1},
 Here we derive the unconditional weighted Marcinkiewicz discretization theorem for $L_p$ norms of functions from a general $N$-dimensional subspace $X_N\subset L_p$.
 We need  the following result.  %from functional analysis:

%The result proved in the last section (i.e., Theorem \ref{thm-6-2}) together with the change of density argument allows us to establish the following unconditional weighted discretization result for  $L^p$-space for $1\leq p\leq 2$.
%
%
%\begin{thm}\label{thm-2-1} Let $1\leq p\leq 2$, and let  $X_N$ be   an $N$-dimensional subspace of $L^p(\Og,\mu)$. Then there exist two   universal constants $C_1, C_2>0$,    sets of $m$   points $x_1,\cdots, x_m\in\Og$ and $m$ positive numbers $\ld_1,\cdots, \ld_m>0$  with $m\leq C_1 N\log^4 N$ 	 such that
%	\begin{align*}
%C_2^{-1} 	\|f\|_{L^p(d\mu)} \leq   \Bl(\sum_{j=1}^m\ld_j |f(x_j)|^p\Br)^{\f1p}\leq C_2 \|f\|_{L^p(d\mu)},\   \  \forall f\in X_N.
%	\end{align*}
%\end{thm}
%

\begin{lem}%\cite[Lemma 7.1]{BLM}
\label{lem-7-2}
{ \textnormal{\cite[Lemma 7.1]{BLM}} }
	Let $X_N$ be an $N$-dimensional subspace of $L_p(\Og, d\mu)$ with  $1\leq p<\infty$. Then there is a basis $\{\vi_j\}_{j=1}^N$ of $X_N$ so that the function
	$ F=\Bl( \sum_{j=1}^N \vi_j ^2\Br)^{\f12}$ satisfies that $ \|F\|_p=1$ and
	for all scalars $\{\ld_j\}_{j=1}^N\subset \RR$,
	\begin{equation}\label{7-1}
	\int_{\Og} \Bl|\sum_{j=1}^N \ld_j \vi_j(x)\Br|^2
	F(x)^{p-2}\, d\mu(x) =N^{-1}\sum_{j=1}^N \ld_j^2.
	\end{equation}
	
\end{lem}

\begin{proof}[Proof of Theorem \ref{thm-2-1}]
	Let $1\leq p\leq 2$, and let $\{\vi_j\}_{j=1}^N$ be a  basis of $X_N$  for which the function $ F=\Bl( \sum_{j=1}^N \vi_j ^2\Br)^{\f12}$ has  the properties in   Lemma \ref{lem-7-2}.
	Then  $d\nu: =F^p d\mu$ is a probability measure on $\Og$.   Define the mapping $U: L_p(d\mu)\to L_p(d\nu)$ by
	$$ Uf (x)=\begin{cases}
	\f {f(x)} {F(x)},\  \ &\text{if $F(x)\neq 0$};\\
	0,\   \  &\text{otherwise}.
	\end{cases}$$
	Note that if $x\in \Og$ and  $F(x)=0$,  then $f(x)=0$ for all $f\in X_N$.  It follows that
	\begin{align}\label{vspom}
\|Uf\|_{L_p(d\nu)} =\|f\|_{L_p(d\mu)},\   \ \forall f\in X_N.
	\end{align}
	Next, let
	$\psi_j =\sqrt{N} U \vi_j$ for $1\leq j\leq N$. Then \eqref {7-1} implies that $\{\psi_j\}_{j=1}^N$ is an orthonormal basis of the space $\wt{X}_N:=U X_N$ equipped with the norm of $L_2(\Og, d\nu)$. Moreover, for any $x\in \Og$ with $F(x)\neq 0$,
	\begin{align}\label{7-3}
	\f 1N \sum_{j=1}^N \psi_j(x)^2 =\f 1{F(x)^2} \sum_{j=1}^N \vi_j(x)^2 =1.
	\end{align}
	Note that \eqref{7-3} with $=1$ replaced by $\leq 1$  holds trivially if $F(x)=0$, in which case $\psi_j(x)=0$ for all $1\leq j\leq N$. This implies that
	$$\|g\|_{L_\infty(\Og)} \leq N^{\f12}\|g\|_{L_2(d\nu)},\   \ \forall g\in \wt{X}_N.$$
	Finally, applying Theorem \ref{thm-6-2} to the space $\wt{X}_N\subset L_p(\Og, d\nu)$ with $K_1=1$, for any $\va\in (0,1)$,  we can find  a set of $m$  points $x_1,\ldots, x_m\in\Og$ with $m\leq C_p(\va)  N\log^3 N$ such that $F(x_j)>0$ for all $j=1,\ldots, m$ and
	$$(1-\va) \|U f\|^p_{L_p(d\nu)}\leq  \f 1m \sum_{j=1}^m |U f(x_j)|^p \leq (1+\va) \|Uf\|^p_{L_p(d\nu)},\   \ \forall f\in X_N.$$
	To complete the proof, we just need to observe that
% $\|Uf\|_{L_p(d\nu)}=\|f\|_{L_p(d\mu)}$	and
	$ Uf(x_j) =\f {f(x_j)}{F(x_j)}$, $j=1,\ldots, m$ 	
	for all $f\in X_N$ and recall \eqref{vspom}.

\end{proof}

\section{Discussion}\label{sec:8}

% We pointed out in the Introduction that this paper is a follow up to the recent papers, especially, to the paper \cite{DPSTT}.
 In  \cite{DPSTT} we have recently proved a conditional result,  Theorem \ref{T2.1}. This theorem  guarantees the existence of good Marcinkiewicz-type discretization results for a given $N$-dimensional subspace $X_N$ under condition on the behavior of the entropy numbers of the unit $L_p$-ball of $X_N$ in the uniform norm $L_\infty$. In this paper we concentrate on establishing good upper bounds for the corresponding entropy numbers $\e_k(X^p_N,L_\infty)$.
 %We have presented results for the case $1\le p\le 2$ here and we plan to discuss the case $2<p<\infty$ in our future papers.
  The problem of estimating the entropy numbers of
 different compacts, including function classes, is a deep fundamental problem of analysis (see, e.g., \cite{Tbook}, \cite{DTU}, and \cite{VTbookMA}).
 % There is a number of open problems in this area. For instance, the reader can find some of them, related to the entropy numbers of classes of functions with mixed smoothness, in \cite{Tbook}, \cite{DTU}, and \cite{VTbookMA}.
  In this section we  discuss some known techniques and results and compare them with our new results. We concentrate on the case of
 entropy numbers in the $L_\infty$ norm. The first step of our technique is the following well-known result (see \cite{PTY}).
 %, for instance, \cite[p.151]{Tbook}).

 \begin{thm}\label{Theorem 32.3} Let $X$ be $\bR^n$ equipped with $\|\cdot\|$ and
$$
M_X = \int_{\SS^{n-1}} \|x\|d\sigma(x).
$$
Then we have
$$
 \e_k(B^n_2,X) \leq C  M_X \left\{\begin{array}{ll}   (n/k)^{1/2}, & k\le n\\
 2^{-k/n} , &  k\ge n.\end{array} \right.
$$
\end{thm}
Theorem \ref{Theorem 32.3} is a dual version of the corresponding result from \cite{Su}. % Theorem \ref{Theorem 32.3} was proved in \cite{PTY}.

%Note that there is
Another technique to estimate the entropy numbers in the $L_\infty$ norm (see \cite{VT156}, \cite{VT159}, \cite{VTbookMA})
is based on the connection between the entropy numbers and the best $m$-term approximations with respect to a dictionary
 and it does not use Theorem \ref{Theorem 32.3}.
 %That technique (see \cite{VT156}, \cite{VT159}, \cite{VTbookMA}) is based on the connection between the entropy numbers and the best $m$-term approximations with respect to a dictionary.
  In estimating the best $m$-term approximations greedy-type algorithms are used. In a certain sense this technique provides a constructive way of building good point nets. This technique was applied  in \cite{VT159} to prove
Marcinkiewicz-type discretization results in the case of $L_1$ norm.

Theorem \ref{Theorem 32.3} is used for estimation of $\e_k(X_N^2,L_\infty)$ (see Subsection \ref{subsection:3.1} above).
To obtain bounds for $\e_k(X_N^p,L_\infty)$, $1\le p<2$,  in contrast with \cite{Bel},
we use in Subsection \ref{subsection:3.3} a new
technique, which allows us to derive bounds of $\e_k(X_N^p,L_\infty)$, $1\le p<2$, from the bounds for $\e_k(X_N^2,L_\infty)$.

%The next step is to obtain bounds for $\e_k(X_N^p,L_\infty)$, $1\le p<2$. In the paper (see Subsection \ref{subsection:3.3}) we use a somewhat elementary technique, which allows us to derive bounds of $\e_k(X_N^p,L_\infty)$, $1\le p<2$, from the bounds for $\e_k(X_N^2,L_\infty)$. It is a new technique. In the paper \cite{Bel} this step is based on deep results from the Banach space theory.

In  \cite{DPSTT} we discussed %presented a discussion of
  some results on the entropy numbers of
the $L_q$-balls of subspaces $\Tr(Q_n)$ of trigonometric polynomials with frequencies from
the hyperbolic cross $Q_n$. We continue this discussion with one more example here. We remind some notation. For $\bs\in\Z^d_+$,
define
$$
\rho (\bs) := \{\bk \in \Z^d : [2^{s_j-1}] \le |k_j| < 2^{s_j}, \quad j=1,\dots,d\},
$$
where $[x]$ denotes the integer part of $x$. We define the step hyperbolic cross
$Q_n$ as follows
$$
Q_n := \cup_{\bs:\|\bs\|_1\le n} \rho(\bs)
$$
and the corresponding set of the hyperbolic cross polynomials as
\begin{align*}
\Tr(Q_n) :&= \{f: f=\sum_{\bk\in Q_n} c_\bk e^{i(\bk,\bx)}, \ c_{\bk}\in\RR\},\\
 \Tr(Q_n)^q:&=\{f\,:\, f\in \Tr(Q_n), \|f\|_q\le 1\}.
\end{align*}
It is well known (see, for instance, \cite{DPSTT}) that the bound
$$
\e_N(X^p_N,L_q) \le B
$$
implies the bound
$$
\e_k(X_N^p, L_q)\le 6B2^{-k/N} \quad k> N.
$$
Thus, we only compare bounds for $1\le k\le N$. The following bound is known (see \cite[p.350]{VTbookMA}) for all dimensions $d$: with  $N:=|Q_n|$ and $\beta:=1/p-1/q$,
\be\label{D1}
\e_k(\Tr(Q_n)^p,L_q) \le C(q,p,d) (N/k)^\beta(\log (N/k))^\beta,\quad 1<p\le 2\le q<\infty.
\ee
Moreover, for $d=2$ one has %and the following one is known
 (see \cite[p.361]{VTbookMA}) %for $d=2$
\be\label{D2}
\e_k(\Tr(Q_n)^1,L_q) \le C(q,p,d) (N/k)^\beta(\log (N/k))^\beta,\quad 2\le q<\infty.
\ee

To demonstrate the strength of our technique, we derive
 %We now demonstrate the power of our new technique. Below, we prove
Theorem \ref{thm-8-1},
which applies to more general subspaces than the $\Tr(Q_n)$ and, moreover, %. Moreover,  Theorem \ref{thm-8-1}
  covers the cases $p=1$ and $q=\infty$. However, we point out that in some cases bounds \eqref{D1} and \eqref{D2} are better for proving the upper bounds for the classes of functions with mixed smoothness. Typically, the extra factor $(\log (N/k))^\beta$ in \eqref{D1} and \eqref{D2} does not contribute in the final upper bound for a class, while the extra factor $(\log N)^\beta$ in Theorem \ref{thm-8-1} will increase the power of the corresponding log factor by $\beta$.

In this section,  we  show that  the following extension of Theorem \ref{thm-4-1} can be proved.
\begin{thm}\label{thm-8-1}
	If $X=X_N$ is an $N$-dimensional subspace of $L_\infty$ satisfying the conditions \eqref{4-1} and \eqref{2-2},
	then for $1\leq p\leq 2<q\leq \infty$ and $\beta=1/p-1/q$,
	\begin{equation}\label{8-1}
	\e_k(X^p; L_q)\leq C_{p,q} (K_1 K_2^2\log N) ^{\beta} \begin{cases} \bl( \f N k \br)^{\beta}, &  \  \  \ \text{if $1\leq k\leq N$};\\
	2^{-k/N}, &\  \ \text{if $k\ge N$}.	
	\end{cases}
	\end{equation}
\end{thm}

\begin{proof} Since the case $q=\infty$ is contained in Theorem \ref{thm-4-1}, we assume $q<\infty$.
 First, we observe  that
\begin{equation}\label{4-17-0}
\cH_\va (X_N^2; L_q) \leq C_q K_1 K_2^2 {\va^{-\f{2q}{q-2}}} { N \log N},\   \ 2<q < \infty, \  \va>0.
\end{equation}
Indeed, setting  $\ta=\f 2q$, we have that   for any $f, g\in X_N^2$,
$$ \|f-g\|_q \leq \|f-g\|_2^\ta \|f-g\|_\infty^{1-\ta} \leq 2^\ta \|f-g\|_\infty^{1-\ta}. $$
Thus, using  \eqref{infinity}, we obtain
$$\cH_\va (X_N^2; L_q) \leq \cH _{ (2^{-\ta} \va)^{\f 1{1-\ta}}} (X_N^2;\  \  L_\infty) \leq\f {C_q K_1 K_2^2 N\log N} { \va^{\f{2q}{q-2}}}.$$

	Next, we apply   inequality  \eqref{0-9} in Lemma \ref{L1L2ineq} to obtain that
		\begin{align*}
		\cH_{\va} (X_N^p; L_{q}) \le \sum_{s=0}^\infty \cH_{ 2^{-3}a^{s-1} \va^{\ta} } (X_N^2; L_{q})+\cH_{\va^\ta} (X_N^2; L_{q}),\   \  \va>0,
		\end{align*}
		where  	$\ta:=(\f 12-\f 1q)/(\f 1p-\f1q)$ and $a=a(\ta) =2^{\f \ta {1-\ta}}$.
	Thus, we derive from \eqref{4-17-0}	that
\begin{align*}
\cH_{\va} (X^p; L_{q}) \le C_q K_1K_2^2 N\log N \va^{-1/(\f 1p-\f1q)},
\end{align*}
which implies \eqref{8-1}.

\end{proof}

The primary goal of this paper is to obtain the Marcinkiewicz-type discretization results with equal weights (see \eqref{I.1}) for a wide class of finite dimensional subspaces $X_N$. %The strongest result of the paper in this direction is Theorem \ref{thm-6-2}.
 In Theorem \ref{thm-6-2} above, we only impose one restriction on a subspace $X_N$, namely, the Nikol'skii inequality \eqref{4-1}: for any $f\in X_N$,
\be\label{D3}
\|f\|_\infty \le (K_1N)^{1/2}\|f\|_2.
\ee
Under that assumption and a minor assumption on $K_1$, Theorem \ref{thm-6-2} guarantees that   $X_N$ admits the Marcinkiewicz-type discretization theorem with $m$ of order $N(\log N)^3$. It is clear that our result is optimal with respect to power scale.
However, it would be interesting to know if it is possible to  replace in Theorem \ref{thm-6-2} the bound $m\le C_p(\alpha)K_1 N(\log N)^3$ by
the bound $m\le C_p(\alpha)K_1 N$.

 Another question  is to what extent we can weaken  the Nikol'skii inequality \eqref{D3} and still have Theorem \ref{thm-6-2} with the bound $m\le C_p(\alpha)K_1 N(\log N)^c$?
  Here we stress that  Theorem \ref{thm-2-1} above shows that
we can drop the assumption on the Nikol'skii inequality if we allow arbitrary weights instead of
equal weights in the discretization theorem. We now make some comments on Theorem \ref{thm-2-1}.
%  There are two interesting open problems in this regard.

%{\bf Open problem 1.} Could we replace in Theorem \ref{thm-6-2} the bound $m\le C_p(\alpha)K_1 N(\log N)^3$ by
%the bound $m\le C_p(\alpha)K_1 N$?

%Certainly, it would also be interesting to replace in Theorem \ref{thm-6-2} factor $(\log N)^3$ by a slower growing function on $N$.

%{\bf Open problem 2.} How much can we weaken the condition that the Nikol'skii inequality \eqref{D3} holds and still have Theorem \ref{thm-6-2} with the bound $m\le C_p(\alpha)K_1 N(\log N)^c$?

%Keeping these questions  in mind, we present Theorem \ref{thm-2-1}, which shows that
%we can drop the assumption on the Nikol'skii inequality if we allow arbitrary weights instead of
%equal weights in the discretization theorem. We now make some comments on Theorem \ref{thm-2-1}.

We consider two cases: $p=2$ and $1\le p<2$.
 %First, let us discuss the case $p=2$.
  In the case $p=2$, the recent results from
\cite{BSS} basically solve the discretization problem with weights; see \cite{VT159}.
% is pointed out in \cite{VT159} that a recent paper %  result by J. Batson, D.A. Spielman, and N. Srivastava
% \cite{BSS} basically solves the discretization problem with weights in case $p=2$.
 We present an explicit formulation of this important result in our notations.
\begin{theorem*}
	
Let  $\Omega_M=\{x^j\}_{j=1}^M$ be a discrete set with the probability measure $\mu(x^j)=1/M$, $j=1,\dots,M$ and let $X_N$ be an $N$-dimensional subspace of real functions defined on $\Omega_M$.
Then for any
number $b>1$ there exist a set of weights $\lambda_j\ge 0$ such that $|\{j: \lambda_j\neq 0\}| \le bN$ so that for any $f\in X_N $ we have
\be\label{C2'}
\|f\|_2^2 \le \sum_{j=1}^M \lambda_jf(x^j)^2 \le \frac{b+1+2\sqrt{b}}{b+1-2\sqrt{b}}\|f\|_2^2.
\ee
\end{theorem*}
As was observed  in \cite [Theorem 2.13]{DPTT},  this last theorem with  a general probability space $(\Og, \mu)$ in place of the discrete space  $(\Og_M, \mu)$ remains true if $X_N\subset L_4(\Og)$.  We further remark here that the additional assumption $X_N\subset L_4(\Og)$ can be dropped as well; namely, we have
\begin{thm} If $X_N$ is an $N$-dimensional subspace of $L_2(\Og)$, then for any $b\in (1,2]$, there exist a set of $m\leq bN$ points $x^1,\ldots, x^m\in\Og$ and a set of nonnegative  weights $\ld_j$, $j=1,\ldots, m$  such that
\[ \|f\|_2^2\leq  \sum_{j=1}^m \ld_j |f(x^j)|^2 \leq  \f C{(b-1)^2}  \|f\|_2^2,\  \ \forall f\in X_N,\]
where $C>1$ is an absolute constant.
\end{thm}

\begin{proof}
Let $\{\vi_j\}_{j=1}^N$ be an orthonormal basis  of $X_N$, and let  $$ F(x):=\Bl(\f 1N \sum_{j=1}^N \vi_j(x) ^2\Br)^{\f12}, \  \ x\in\Og.$$
Consider the  probability measure  $d\nu: =F^2 d\mu$  on  the set $$\Og_0=\{x\in\Og:\  \ F(x)>0\}.   $$
       Define the mapping $U: L_2(\Og, d\mu)\to L_2(\Og_0, d\nu)$ by $$Uf(x)=\f {f(x)}{F(x)},\  \   x\in\Og_0,$$
        and let $\wt{X}_N:=U X_N$ be the  subspace of $L_2(\Og_0, d\nu)$.
Then
\begin{equation}\label{0-2b}
\|Uf\|_{L_2(\Og_0, d\nu)} =\|f\|_{L_2(\Og,d\mu)},\   \ \forall f\in X_N,	\end{equation}
and  by the  Cauchy-Schwarz inequality,
$$\|g\|_{L_\infty(\Og_0,d\nu)} \leq N^{\f12}\|g\|_{L_2(\Og_0,d\nu)},\   \ \forall g\in \wt{X}_N.$$
Now applying Theorem \ref{T5.4}   to the space $\wt{X}_N\subset L_2(\Og_0, d\nu)$,  we obtain   a finite subset $\Ld\subset \Og_0$   such that $|\Ld|\leq C N\log N$ and
$$\f 12 \|U f\|^2_{L_2(\Og_0, d\nu)}\leq  \f 1{|\Ld|} \sum_{\og\in\Ld} |U f(\og)|^2 \leq \f 32 \|Uf\|^2_{L_2(\Og_0,d\nu)},\   \ \forall f\in X_N.$$
It then follows by \eqref{0-2b} that for any $f\in X_N$,
\begin{equation}\label{0-3b}
\f 12 \| f\|^2_{L_2(\Og, d\mu)}\leq  \f 1{|\Ld|} \sum_{\og\in\Ld} |f(\og)|^2(F(\og))^{-2} \leq \f 32 \|f\|^2_{L_2(\Og,d\mu)}.
\end{equation}

Finally, applying \eqref{C2'} to the subspace  $Y_N:=\big\{\f fF\big|_{\Ld}:\  \ f\in X_N\big\}$ of $\RR^{| \Ld|}$, we conclude that for any $b\in (1,2]$, there exist a set of $m$ points $x^1,\ldots, x^m\in \Ld\subset \Og_0$   and a set of weights $w_j \ge 0$, $1\leq j\leq m$ such that  $m\leq b N$ and
\begin{align*}
\f 1 {|\Ld|} &\sum_{\og\in \Ld} |f(\og)|^2 (F(\og))^{-2} \leq \sum_{j=1}^m  w_j |f(x^j)|^2 F(x^j)^{-2}\\
& \leq C (b-1)^{-2} \f 1 {|\Ld|} \sum_{\og\in \Ld} |f(\og)|^2 (F(\og))^{-2},\  \ \forall f\in X_N.
\end{align*}
This together with \eqref{0-3b} implies that
\begin{align*}
\| f\|^2_{L_2(\Og, d\mu)}\leq 2\sum_{j=1}^m w_j |f(x^j)|^2 F(x^j)^{-2}\leq 3C (b-1)^{-2} \| f\|^2_{L_2(\Og, d\mu)}.
\end{align*}
This completes the proof of the theorem with $\ld_j=w_j/F(x^j)^2$.

\end{proof}

Finally, let us  discuss the case $1\le p<2$. Our proof of Theorem \ref{thm-2-1} demonstrates
 how deeps results from functional analysis (see Lemma \ref{lem-7-2}) can be used in obtaining good discretization theorems with weights. That kind of technique was developed in
 the following important problem from functional analysis (see, for instance, \cite{BLM}, \cite{Sche3}, \cite{JS}): Given an $N$ dimensional subspace $X_N$ of $L_q(0,1)$ and $\e>0$, what is the
smallest $L(X_N,q,\e)$ such that there is a subspace $Y_N$ of $\ell^{L(X_N,q,\e)}_q$ with $d(X_N,Y_N)\le 1+\e$?  Here $d(X,Y)$ stands for the Banach--Mazur distance
between two finite dimensional spaces $X$ and $Y$ of the same dimension, that is, % is defined to be
$$
d(X,Y):= \inf\{\|T\|\|T^{-1}\|;\, T\;\, \text{is a linear isomorphism from} \, X\,\text{to}\, Y\}.
$$
%The following problem has been studied in functional analysis: Given an $N$ dimensional subspace $X_N$ of $L_q(0,1)$ and $\e>0$, what is the
%smallest $L(X_N,q,\e)$ such that there is a subspace $Y_N$ of $\ell^{L(X_N,q,\e)}_q$ with $d(X_N,Y_N)\le 1+\e$? It
This question is related to the following discretization problem.
\\
{\bf Marcinkiewicz problem with $\e$.} We write $X_N\in \cM(m,q,\e)$ if (\ref{I.1}) holds with $C_1(d,q)=1-\e$ and $C_2(d,q)=1+\e$.

Certainly, if $X_N\in \cM(m,q,\e)$ then $L(X_N,q,C(q)\e)\le m$. On the other hand, results on the behavior of  $L(X_N,q,\e)$ do not imply bounds on $m$ for
$X_N\in \cM(m,q,\e)$. It is obvious in the case $q=2$. However, it turns out that the technique
developed in this area for studying behavior of $L(X_N,1,\e)$ can be used for the Marcinkiewicz-type discretization (see \cite{BLM}, \cite{Sche3}, \cite{JS}).
%Moreover, some papers on that topic (see \cite{BLM}, \cite{Sche3}, \cite{JS}) contain implicitly some discretization results.
 %Theorem \ref{thm-2-1} is only an illustration of application of an element of that deep technique.
 %We plan to continue to work on further applications of that technique in discretization.

  \section*{\bf Acknowledgement} The authors are extremely grateful to the  Isaac Newton Institute  (Cambridge, UK): their collaboration has
  started while they  participated in the research program ``Approximation, sampling, and compression in high dimensional problems''  in 2019.
  The work was supported by the Russian Federation Government Grant N{\textsuperscript{\underline{o}}}14.W03.31.0031. The paper contains results obtained in frames of the program
 ``Center for the storage and analysis of big data", supported by the Ministry of Science and High Education of Russian Federation (contract 11.12.2018N{\textsuperscript{\underline{o}}}13/1251/2018 between the Lomonosov Moscow State University and the Fond of support of the National technological initiative projects).

 \Addresses

\end{document}